\documentclass[a4paper, 12pt]{article}
\usepackage[a4paper, total={6in,10in}]
{geometry}
\usepackage[utf8]{inputenc}
\usepackage{fullpage}
\usepackage{amsmath}
\usepackage{amssymb}
\usepackage{amsfonts}
\usepackage{amsthm}
\usepackage{graphicx}
\usepackage{xcolor}
\usepackage{float}
\usepackage{cite}
\usepackage{algorithm2e}
\usepackage{enumitem}
\usepackage{mathrsfs}
\newtheoremstyle{dotless}{}{}{\itshape}{}{\bfseries}{}{ }{}
\theoremstyle{dotless}
\newtheorem{theorem}{Theorem}[section]
\newtheorem{corollary}[theorem]{Corollary}
\newtheorem{lemma}[theorem]{Lemma}

\theoremstyle{definition}

\title{Characterizing tricyclic graphs with pendant vertices having largest $A_{\alpha}$-spectral radius}

\author{Mainak Basunia\thanks{Department of Mathematics, Indian Institute of Technology Kharagpur, Kharagpur 721302, India. Email: mainakmaths@iitkgp.ac.in, leo28mynnix@gmail.com}\and Pratima Panigrahi\thanks{Department of Mathematics, Indian Institute of Technology Kharagpur, Kharagpur 721302, India. Email: pratima@maths.iitkgp.ac.in}}

\date{}

\begin{document}
\maketitle
\baselineskip=0.25in

\begin{abstract}
For a graph $G$ with adjacency matrix $A(G)$ and degree diagonal matrix $D(G)$, the $A_{\alpha}$-matrix of $G$ is defined as
\begin{equation*}
    A_{\alpha}(G) = \alpha D(G) + (1- \alpha) A(G), \text{ for any } \alpha \in [0,1].
\end{equation*}
The $A_{\alpha}$-spectral radius of $G$ is the largest eigenvalue of the matrix $A_{\alpha}(G)$. A tricyclic graph of order $n$ is a simple connected graph with $n+2$ edges. In this paper, we characterize the unique graph having the largest $A_{\alpha}$-spectral radius for $\alpha \in [\frac{1}{2}, 1)$ among all tricyclic graphs of order $n$ with $k (\geq 1)$ pendant vertices. As an application, we derive a sufficient spectral condition (alternate to the edge condition) to guarantee the absence of the tricyclic structure in a graph with $k$ pendant vertices. 


\end{abstract}


\section{Introduction}\label{sec1}
All graphs considered in this paper are simple, undirected and connected. Let $G=$ $(V(G)$, $E(G))$ be a graph with vertex set $V(G)$ and edge set $E(G)$. We denote the number of vertices (order) and the number of edges (size) of $G$ by $n$ and $m$, respectively. The \textit{neighbourhood} of a vertex $v$ in $G$, denoted by $N_G(v)$ (or, $N(v)$), is the set of all vertices in $G$ adjacent to $v$. The \textit{degree} of a vertex $v$, denoted by $d_G(v)$ (or, $d(v)$), is defined as the number $|N_G(v)|$. The \textit{adjacency matrix} of $G$, denoted by $A(G)$, is the square symmetric matrix of order $n$ whose $(i,j)^{th}$ entry is 1 or 0, according to the vertices corresponding to the $i^{th}$ row and $j^{th}$ column are adjacent or not. Let $D(G)$ be the diagonal matrix of order $n$ having the degrees of the vertices of $G$ as its diagonal entries. The matrices $L(G) = D(G) - A(G)$ and $L_S(G) = D(G) + A(G)$ are called the \textit{Laplacian matrix} and \newpage\noindent \textit{signless Laplacian matrix} of $G$, respectively. For any real $\alpha \in [0,1]$, Nikiforov \cite{A_Q-merging_by_Nikiforov} defined the \textit{$A_{\alpha}$-matrix} of $G$ as
\begin{equation}\label{eq1}
A_{\alpha}(G) = \alpha D(G) + (1- \alpha) A(G),\qquad \alpha \in [0,1].
\end{equation}
Clearly, when we take $\alpha = 0, \frac{1}{2}$ and $1$, $A_{\alpha}(G)$ becomes $A(G)$, $\frac{1}{2}L_S(G)$ and $D(G)$, respectively. The \textit{$A_{\alpha}$-spectral radius} of $G$, which is represented as $\rho_{\alpha}(G)$, is the maximal eigenvalue of $A_{\alpha}(G)$. For a connected graph $G$, it follows from \cite{A_Q-merging_by_Nikiforov} that $A_{\alpha}(G)$ is irreducible, $\rho_{\alpha}(G)$ has multiplicity one, and there exists a positive unit eigenvector $X$ (unique up to scaling) corresponding to $\rho_{\alpha}(G)$, which is referred to as the \textit{Perron vector} of $A_{\alpha}(G)$. Throughout the paper, for any vertex $v \in V(G)$, the component of $X$ corresponding to $v$ is denoted by $x_v$. 

The challenge of identifying the graphs for which the maximum or the minimum spectral invariants are achieved in a certain class of graphs was put out by Brualdi and Solheid \cite{spec_rad_by_brualdi_solheid} in $1986$. Since then, the problem has been investigated over a wide range of graph classes and spectral invariants (including spectral radius, Laplacian spectral radius, and signless Laplacian spectral radius); see \cite{spec_rad_of_unicyclic_bicyclic_pendant_by_guo, signless_lap_spec_rad_of_uni_bicyclic_given_girth_by_li_wang_zhou, spec_rad_of_unicyclic_fixed_girth_by_li_guo_shiu, signless_lap_spec_radii_of_c_cyclic_with_pendant_by_liu, signless_lap_spec_radius_of_unicyclic_bicyclic_with_pendant_by_liu_tan_liu, lap_spec_rad_of_tricyclic_pendant_by_guo_wang, spec_rad_of_trees_pendant_by_wu_xiao_hong, max_spec_rad_bicyclic_fixed_girth_by_zhai_wu_shu, Ext_AM_GM_spec_rad_unicyclic_by_niu_zhou_zhang}. The relative findings for $A_{\alpha}$-spectral radius are available in \cite{graphs_of_fixed_order_size_with_max_a_alpha_index_by_chang_tam, spectra_of_join_by_basunia, A_alpha_spec_rad_of_complement_of_bicyclic_tricyclic_by_chen, A_alpha_spec_rad_of_given_size_by_chen_li_huang, A_alpha_spec_rad_given_size_diameter_by_feng_wei, some_spec_inequalities_bipartite_A_alpha_by_li_sun, some_bounds_on_A_alpha_index_fixed_order_size_by_li_sun, alpha_index_of_minimally_2_connected_by_lou_wang_yuan, alpha_index_with_pendant_vertices_by_nikiforov_rojo, two_upper_bounds_on_A_alpha_spec_rad_by_pirzada, sharp_upper_bounds_of_A_alpha_spec_rad_of_cacti_by_wang, on_max_alpha_spec_rad_given_matching_number_by_yuan_shao, clique_trees_zero_forcing_by_jin_li_hou, max_A_alpha_index_with_size_domination_number_by_zhang_guo, ordering_A_alpha_by_wei_feng, some_ext_problems_A_alpha_by_ye_guo_zhang, results_A_alpha_eigenvalues_line_graphs_by_da_silva, max-alpha_of_trees_pendant_by_rojo}.

A \textit{$c$-cyclic graph} of order $n$ and size $m$ is a connected graph such that the relationship $m = n + c -1$ holds true. Specifically, when $c$ takes on the values of $0, 1, 2$, or $3$, the graph is categorized as a \textit{tree, unicyclic, bicyclic}, or \textit{tricyclic} graph, respectively. In $2018$, Lin et al. \cite{A_alpha_spec_rad_by_lin} provided a characterization of graphs that exhibit the largest $A_{\alpha}$-spectral radii among all trees with a specified order and matching number. In $2019$, Li et al. \cite{A_alpha_spec_rad_trees_unicyclic_degree_seq} and in $2023$, Wen et al. \cite{A_alpha_spec_rad_of_bicyclic_degree_seq_by_wen} conducted investigations into the extremal problems that are associated with $A_{\alpha}$-spectral radius of unicyclic and bicyclic graphs, respectively, with the constrain that in each of the collections of graphs, all graphs have a specified degree sequence. Out of all unicyclic (resp. bicyclic) graphs that possess a fixed diameter, the graphs having the highest $A_{\alpha}$-spectral radii were identified by Wang et al. \cite{Alpha_spec_rad__of_unicyclic_bicyclic_fixed_diameter_by_wang}. Recently, Das and Mahato \cite{On_max_a_alpha_spec_rad_of_unicyclic_bicyclic_fixed_girth_pendant_by_das_mahato} presented characterization of graphs with maximal $A_{\alpha}$-spectral radius within the class of unicyclic graphs, as well as within the class of bicyclic graphs, under the constraints of fixed girth or a prescribed number of pendant vertices. 

Inspired by these works, we extend our focus to tricyclic graphs, which represent the next level of $c$-cyclic graphs after trees, unicyclic, and bicyclic classes. Despite significant progress in related settings, no such complete characterization of the graphs attaining the maximum $A_\alpha$-spectral radius exists for tricyclic graphs with pendant vertices. The presence of additional paths attached to the graphs plays a crucial role in determining the $A_\alpha$-spectral radius, as their location around a central vertex strongly influences the components of the Perron vector. Thus, characterizing the extremal graph in this particular family continues the earlier line of research while also providing a genuinely new contribution to the literature.

Beyond the purely theoretical setting, our work  is motivated by practical applications also. It is known that spectral radius of a graph governs important dynamical processes on networks such as synchronization, diffusion, and epidemic spreading; see \cite{Appx_algo_by_saha, Applications_by_nardo}. In chemical graph theory, tricyclic graphs arise naturally as models for molecular structures, and pendant vertices may be interpreted as side chains, while the spectral radius is often linked to molecular stability and reactivity; see \cite{Sharp_tricyclic_isi_index_by_zhang, Characterization_by_kour}. Identifying the graphs with maximum $A_\alpha$-spectral radius among the tricyclic graphs with pendant vertices therefore contributes not only to the general program of extremal spectral graph theory but also provides insights relevant in applied contexts where pendant-like attachments and tricyclic cores occur.

In order to formulate the main results of the paper, we need some definitions. A number of paths are said to be of nearly equal lengths if their pairwise lengths differ by at most $1$. For two positive integers $n$ and $k$, let $\mathscr{T}_n^k$ be the set of all tricyclic graphs of order $n$ with $k$ pendant vertices. We define $\mathcal{T}_3 \in \mathscr{T}_n^k$ as the graph obtained from three triangles sharing exactly one common vertex, by attaching $k$ paths of nearly equal lengths to that common vertex (see Figure \ref{f9}).
\begin{figure}[t]
    \centering
    \includegraphics[scale=.7]{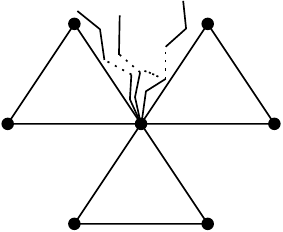}
    \caption{The graph $\mathcal{T}_3$.}
    \label{f9}
    \end{figure}

Now we present the principal finding of the paper in Theorem \ref{r6} below. The proof is discussed in the subsequent sections.

\begin{theorem}\label{r6}
    Let $\alpha \in [\frac{1}{2}, 1)$ and $n,k$ be integers with $1\leq k \leq n-7$. Then out of all graphs in $\mathscr{T}_n^{k}$, $\mathcal{T}_3$ stands out as the sole graph to have the largest $A_{\alpha}$-spectral radius.
\end{theorem}

An immediate consequence of Theorem \ref{r6} is Corollary \ref{r7}, which serves as a sufficient condition (alternate condition to the edge criterion) to ensure that a given graph with pendant vertices does not possess tricyclic structure.

\begin{corollary}\label{r7}
    Suppose $G$ is a graph having $k$ pendant vertices. If $G \neq \mathcal{T}_3$, and there exists some $\alpha \in [\frac{1}{2},1)$ for which the condition $\rho_{\alpha} (G) \geq \rho_{\alpha}(\mathcal{T}_3)$ holds, then $G$ is not tricyclic.
\end{corollary}


\section{Preliminaries}
We use the notations $P_n$, $C_n$ and $K_n$ for a path, cycle and complete graph on $n$ vertices, respectively. $K_{a,b}$ represents the complete bipartite graph with $a$ vertices in one partite set, and $b$ vertices in the other. Given a graph $G$, $d(u,v)$ denotes the distance between the vertices $u$ and $v$ in $G$. A path $P$ in $G$ starting at vertex $v_0$ is called a \textit{pendant path} if $d(v_0) \geq 2$, the degree of the other end vertex is $1$, and all the remaining vertices, if any, of $P$ are of degree $2$ in $G$. A tree $T$ attached at $v_0 \in G$ is called a \textit{pendant tree} if $d(v_0) \geq 2$, and all other vertices of $T$ have no neighbors outside $T$. By an \textit{internal path} in this paper, we mean the following two structures in $G$: $(i)$ a path $P$ with degrees of both the end vertices at least $3$, and the remaining vertices, if any, are of degree $2$ each in $G$; $(ii)$ a cycle $C$ having starting and end vertex as $v_0$ with $d(v_0) \geq 3$, and remaining vertices are of degree $2$ each in $G$.

Now, we will look at some of the well-known results from the literature, which we will use to come up with the key findings of the paper.

\begin{lemma}\textup{\cite[Corollary $13$]{A_Q-merging_by_Nikiforov}}\label{lem2}
    If $G$ is a graph with maximum degree $\Delta$, then
    \begin{align*}
         \rho_{\alpha}(G) \geq 
        \begin{cases}
            \alpha (\Delta +1), & \text{ if $\alpha \in [0,\frac{1}{2}]$}\\
            \alpha \Delta + \frac{(1-\alpha)^2}{\alpha}, & \text{ if $\alpha \in [\frac{1}{2},1)$}
        \end{cases}
    \end{align*}
\end{lemma}

\begin{lemma}\textup{\cite[Proposition $14$]{A_Q-merging_by_Nikiforov}}\label{lem1}
    Let $H$ be a proper subgraph of a connected graph $G$, and $\alpha \in [0, 1)$. Then $\rho_{\alpha}(H) < \rho_{\alpha}(G)$.
\end{lemma}

\begin{lemma}\textup{\cite[Lemma $2.2$]{A_alpha_spec_rad_by_xue_lin_liu_shu}}\label{lem3}
    Suppose $G$ is a connected graph and $u,v$ are two of its vertices. Let $N_1$ be the set of those adjacent vertices of $v$ (excluding $u$, if $u$ is a adjacent vertex of $v$), which are not adjacent to $u$. Take any non empty subset $N_2$ of $N_1$. We construct $G^{\prime}$ by removing all the edges between the vertex $v$ and each vertex of $N_2$ in $G$, and then adding edges between $u$ and all the vertices of $N_2$. We denote the $A_{\alpha}$-Perron vector of $G$ by $X$; $x_u$ and $x_v$ are the components of $X$ corresponding to the vertices $u$ and $v$, respectively. If $\alpha \in [0,1)$ and $x_u \geq x_v$, then it follows that $\rho_{\alpha}(G^{\prime}) > \rho_{\alpha}(G)$.
\end{lemma}

\begin{lemma}\textup{\cite[Lemma $1.1$]{A_alpha_spec_rad_trees_unicyclic_degree_seq}}\label{lem4}
    Consider an edge $uv$ located on an internal path in a connected graph $G$. In $G$, we subdivide $uv$ into two edges $uw$ and $wv$; the new graph is represented by $G_{uv}$. If $\alpha \in [0,1)$, then $\rho_{\alpha}(G_{uv}) < \rho_{\alpha}(G)$.
\end{lemma}

\begin{lemma}\textup{\cite[Theorem $3.1$]{A_alpha_spec_rad_by_guo_zhou}}\label{lem5}
    Suppose $G$ is a connected graph having at least one edge, and $r, s$ be two non negative integers. $G(r,s)$ represents the graph we get by attaching two pendant paths of length $r$ and $s$ to any one particular vertex of $G$. If $r \geq s \geq 1$ and $\alpha \in [0,1)$, then it follows that $\rho_{\alpha}(G(r,s)) > \rho_{\alpha}(G(r+1,s-1))$.
\end{lemma}

\begin{lemma}\label{lem6}
     Let $\alpha \in [0,1)$. Consider an edge $uv$ located on a pendant path in a connected graph $G$. In $G$, we subdivide $uv$ into two edges $uw$ and $wv$; the new graph is represented by $G_{uv}$. Then $\rho_{\alpha}(G_{uv}) > \rho_{\alpha}(G)$.
\end{lemma}

\begin{proof}
    Clearly, $G$ is a proper subgraph of the connected graph $G_{uv}$. Applying Lemma \ref{lem1} on the graphs, we directly get the required result.
\end{proof}

\begin{lemma}\textup{\cite[Proposition $20$]{A_Q-merging_by_Nikiforov}}\label{lem7}
    Suppose $G$ is a graph which does not contain any isolated vertices. For any vertex $v$ in $G$, let us define $q(v) = \frac{1}{d(v)}\sum_{uv \in E(G)}d(u)$. If $\alpha \in [0,1)$, then
    \begin{align*}
        \rho_{\alpha}(G) \leq \max_{v \in V(G)}\{\alpha d(v) + (1-\alpha)q(v)\}.
    \end{align*}
    If $G$ is connected and $\alpha \in (\frac{1}{2},1)$, the equality holds if and only if $G$ is regular.
\end{lemma}

\begin{lemma}\textup{\cite[Theorem $3.2$]{bounds_spec_rad_A_alpha_by_alhevaz}}\label{lem8}
    Let $\alpha \in [\frac{1}{2}, 1]$. If $G$ is a connected graph of order $n$ with maximum degree $\Delta$, and $\lambda(L_S(G))$ is the spectral radius of the matrix $L_S(G)$, then
    \begin{align*}
        \rho_{\alpha}(G) \leq (2\alpha -1)\Delta + (1-\alpha)\lambda(L_S(G)).
    \end{align*}
    Equality occurs if and only if $G$ is a non-regular graph and $\alpha = \frac{1}{2}$ or $1$; or $G$ is regular.
\end{lemma}


\section{Proof of Theorem \ref{r6}}

In this section, we establish a series of auxiliary lemmas that will be used to prove our main result, Theorem \ref{r6}.

It is known from \cite{spec_rad_of_tricyclic_by_geng_li} that the total number of cycles a tricyclic graph may contain can be $3$, $4$, $6$, or $7$. Thus, every graph in $\mathscr{T}_n^k$ is derived from some graph in Figure \ref{f1} by connecting trees to some of their vertices. For each $i = 3, 4, 6, 7$, let us denote the collection of tricyclic graphs in $\mathscr{T}_n^k$ that have precisely $i$ number of cycles by $\mathscr{T}_n^{k,i}$. Then $\mathscr{T}_n^k = \bigcup_{i\in \{3,4,6,7\}}\mathscr{T}_n^{k,i}$. 

\begin{figure}[htbp]
    \centering
    \includegraphics[scale=.7]{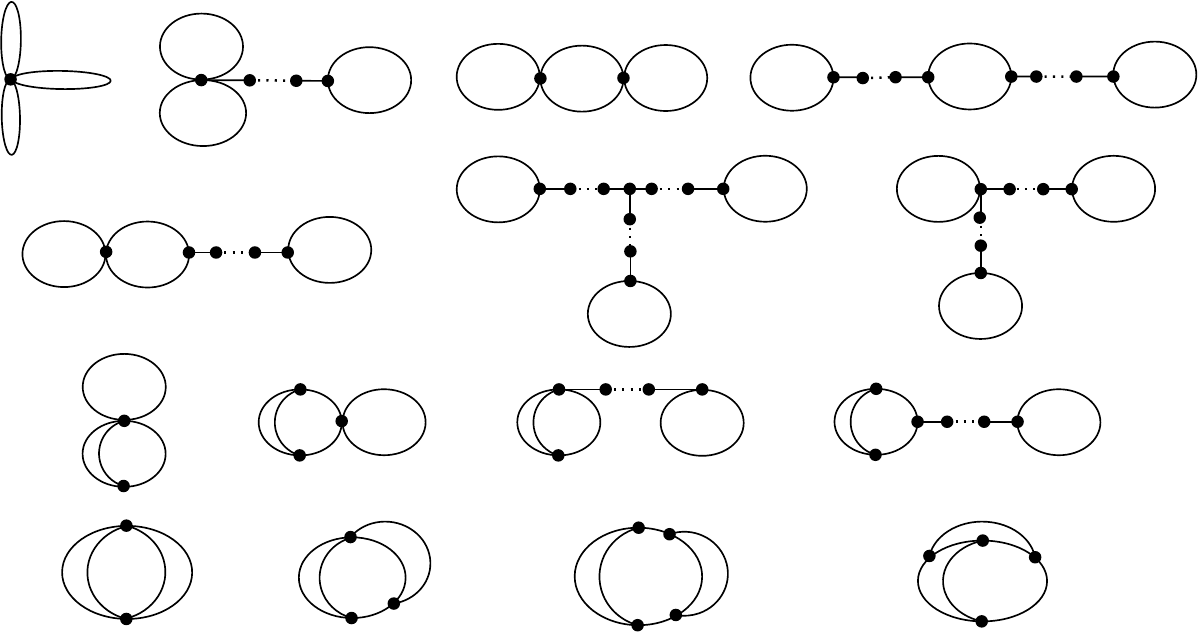}
    \caption{All possible configurations of cycles in graphs of class $\mathscr{T}_n^k$.}
    \label{f1}
\end{figure}

\begin{figure}[htb]
    \centering
    \includegraphics[scale=.6]{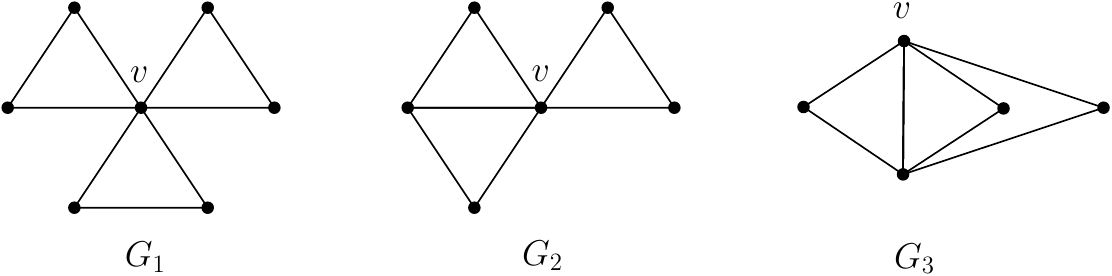}
    \caption{Graphs $G_1$, $G_3$ and $G_3$.}
    \label{f2}
\end{figure}

We recall the definition of the graph $\mathcal{T}_3 \in \mathscr{T}_n^k$ from Section $1$. $\mathcal{T}_3$ is obtained from the graph $G_1$ (see Figure \ref{f2}) after connecting $k$ number of paths having nearly equal lengths to the common vertex $v$ in $G_1$. Similarly, let $\mathcal{T}_4$ and $\mathcal{T}_6$ denote the tricyclic graphs in $\mathscr{T}_n^k$ derived from $G_2$ and $G_3$ (see Figure \ref{f2}), respectively, after connecting $k$ number of paths having nearly equal lengths to the vertex $v$ in each of them. Lastly, let $\mathcal{T}_7$ denote the tricyclic graph in $\mathscr{T}_n^k$ obtained from the graph $K_4$ by connecting $k$ number of paths which have nearly equal lengths, to any vertex $v$ of $K_4$. One can observe that $\mathcal{T}_i \in \mathscr{T}_n^{k,i}$, $i=3, 4, 6, 7$.

In order to prove Theorem \ref{r6}, we now present several lemmas one after another. We begin with the first lemma of this section, which provides a comparison of the $A_{\alpha}$-spectral radii of the four graphs $\mathcal{T}_3$, $\mathcal{T}_4$, $\mathcal{T}_6$ and $\mathcal{T}_7$.

\begin{lemma}\label{r1}
     Let $\alpha \in [\frac{1}{2}, 1)$ and $n,k$ be integers with $1\leq k \leq n-7$. The graph $\mathcal{T}_3$ has strictly largest $A_{\alpha}$-spectral radius among $\mathcal{T}_3$, $\mathcal{T}_4$, $\mathcal{T}_6$ and $\mathcal{T}_7$.
\end{lemma}

\begin{proof}
    To establish the lemma, we must prove a few statements first.
    \medskip\\ 
    \noindent \textbf{Claim 1:} For $\alpha \in [\frac{1}{2}, 1)$, $\rho_{\alpha}(\mathcal{T}_3) > \rho_{\alpha}(\mathcal{T}_7)$.
    \begin{proof}[\unskip\nopunct]\let\qed\relax
        \textbf{Proof of Claim 1:} Maximum degree of $\mathcal{T}_3$ is $k+6$. Therefore by Lemma \ref{lem2}, for $\alpha \in [\frac{1}{2}, 1)$ we have, 
        \begin{align}\label{eq2}
            \rho_{\alpha}(\mathcal{T}_3) \geq \alpha (k+6) + \frac{(1-\alpha)^2}{\alpha}.
        \end{align}
        By Lemma \ref{lem7}, 
        {
        \allowdisplaybreaks
        \begin{align}\label{eq3}
            \rho_{\alpha}(\mathcal{T}_7) &\leq \max\bigg\{\alpha (k+3) +(1-\alpha)\frac{2k+9}{k+3}, 3\alpha + (1-\alpha)\frac{k+9}{3}, 2\alpha + (1-\alpha)\frac{k+5}{2},\notag\\
            &\quad\ 2\alpha + (1-\alpha)\frac{k+4}{2}, \alpha + (1-\alpha)\frac{k+3}{1}, 2\alpha + (1-\alpha)\frac{4}{2}, 2\alpha + (1-\alpha)\frac{3}{2}, \notag\\
            &\quad\ \alpha + (1-\alpha)\frac{2}{1} \bigg\}\notag\\
            &= \max\bigg\{ \alpha (k+3) +(1-\alpha)\frac{2k+9}{k+3}, 3\alpha + (1-\alpha)\frac{k+9}{3}, 2\alpha + (1-\alpha)\frac{k+5}{2}, \notag\\
            &\quad\ \alpha + (1-\alpha)\frac{k+3}{1} \bigg\},
        \end{align}
        }
        by discarding the smaller terms.
        Now $\alpha (k+3) +(1-\alpha)\frac{2k+9}{k+3} = \alpha (k+3) +(1-\alpha)\frac{2(k+9/2)}{k+3} < \alpha (k+3) + (1-\alpha)3$, since $\frac{2(k+9/2)}{k+3}$ is a monotonically decreasing function in $k$ $(k\geq 1)$ and ${\frac{2(k+9/2)}{k+3}\big|}_{k=1}=\frac{11}{4}<3$. Therefore when $\alpha \in [\frac{1}{2},1)$,
        \begin{align}\label{eq4}
            \alpha (k+3) +(1-\alpha)\frac{2k+9}{k+3} < \alpha k +3 \leq \alpha (k+6) + \frac{(1-\alpha)^2}{\alpha}.
        \end{align}
        Again for $\alpha \in [\frac{1}{2},1)$, we have
        \begin{align}\label{eq5}
            3\alpha + (1-\alpha)\frac{k+9}{3} &< 3\alpha +(1-\alpha)(k+3) = \alpha (k+6) + (1-2\alpha)(k+3)\notag\\
            &\leq \alpha (k+6) < \alpha (k+6) + \frac{(1-\alpha)^2}{\alpha}.
        \end{align}
        For $\alpha \in [\frac{1}{2}, 1)$, $2\alpha + (1-\alpha)\frac{k+5}{2} < 3\alpha +(1-\alpha)(k+3)$. Therefore
        \begin{align}\label{eq7}
            2\alpha + (1-\alpha)\frac{k+5}{2} < \alpha (k+6) + \frac{(1-\alpha)^2}{\alpha}.
        \end{align}
        Also for $\alpha \in [\frac{1}{2}, 1)$, $\alpha + (1-\alpha)\frac{k+3}{1} < 3\alpha +(1-\alpha)(k+3)$. Therefore
        \begin{align}\label{eq100}
            \alpha + (1-\alpha)\frac{k+3}{1} < \alpha (k+6) + \frac{(1-\alpha)^2}{\alpha}.
        \end{align}
        Using \eqref{eq4}, \eqref{eq5}, \eqref{eq7} and \eqref{eq100} in \eqref{eq3}, we get
        \begin{align}\label{eq6}
            \rho_{\alpha}(\mathcal{T}_7) < \alpha (k+6) + \frac{(1-\alpha)^2}{\alpha}, \qquad \alpha \in \Big[\frac{1}{2},1\Big).
        \end{align}
        \eqref{eq2} and \eqref{eq6} together prove Claim $1$.
    \end{proof}
    \noindent\textbf{Claim 2:} For $\alpha \in [\frac{1}{2}, 1)$, $\rho_{\alpha}(\mathcal{T}_3) > \rho_{\alpha}(\mathcal{T}_6)$.
    \begin{proof}[\unskip\nopunct]\let\qed\relax
        \textbf{Proof of Claim 2:} By Lemma \ref{lem7}, we have
        \begin{align}\label{eq8}
            \rho_{\alpha}(\mathcal{T}_6) &\leq \max\bigg\{\alpha (k+4) +(1-\alpha)\frac{2k+10}{k+4}, 2\alpha + (1-\alpha)\frac{k+8}{2}, 4\alpha + (1-\alpha)\frac{k+10}{4},\notag\\
            &\quad\ 2\alpha + (1-\alpha)\frac{k+6}{2}, 2\alpha + (1-\alpha)\frac{k+5}{2}, \alpha + (1-\alpha)\frac{k+4}{1}, 2\alpha + (1-\alpha)\frac{3}{2},\notag \\
            &\quad\ \alpha + (1-\alpha)\frac{2}{1}, 2\alpha + (1-\alpha)\frac{4}{2} \bigg\}\notag\\
            &= \max\bigg\{\alpha (k+4) +(1-\alpha)\frac{2k+10}{k+4}, 2\alpha + (1-\alpha)\frac{k+8}{2}, \notag\\
            &\quad\ 4\alpha + (1-\alpha)\frac{k+10}{4}, \alpha + (1-\alpha)\frac{k+4}{1} \bigg\}.
        \end{align}
        In a similar way to the proof of Claim 1, we find $\alpha (k+4) +(1-\alpha)\frac{2k+10}{k+4} = \alpha (k+4) +(1-\alpha)\frac{2(k+5)}{k+4} \leq \alpha (k+4) +(1-\alpha)\frac{12}{5} = \alpha (k+6) +\frac{12}{5} -\frac{22\alpha}{5}$. Now $\frac{(1-\alpha)^2}{\alpha}-\big(\frac{12}{5} -\frac{22\alpha}{5}\big)= \frac{1}{\alpha}-\frac{22}{5} + \frac{27\alpha}{5} = \frac{27\alpha^2 -22\alpha +5}{5\alpha}$. Here, in the numerator, the quadratic polynomial $27\alpha^2 -22\alpha +5$ is positive for any $\alpha$, since its leading coefficient is positive and its discriminant is negative. Therefore $\frac{27\alpha^2 -22\alpha +5}{5\alpha} > 0$ when $\alpha \in [\frac{1}{2}, 1)$. Thus we get
        \begin{align}\label{eq9}
            \alpha (k+4) +(1-\alpha)\frac{2k+10}{k+4} < \alpha (k+6) + \frac{(1-\alpha)^2}{\alpha}, \qquad \alpha \in \Big[\frac{1}{2}, 1\Big).
        \end{align}
        Again $2\alpha + (1-\alpha)\frac{k+8}{2} < 2\alpha + (1-\alpha)(k+4)= \alpha (k+6) +(1-2\alpha)(k+4)\leq \alpha (k+6)$, when $\alpha \in [\frac{1}{2}, 1)$. Therefore
        \begin{align}\label{eq10}
            2\alpha + (1-\alpha)\frac{k+8}{2} < \alpha (k+6) + \frac{(1-\alpha)^2}{\alpha}, \qquad \alpha \in \Big[\frac{1}{2}, 1\Big). 
        \end{align}
        For $\alpha \in [\frac{1}{2}, 1)$, $4\alpha + (1-\alpha)\frac{k+10}{4} = \frac{1}{4}\big[16\alpha +(1-\alpha)(k+10)\big] = \frac{1}{4}\big[\alpha (k+6) +10+(1-2\alpha)k\big] \leq \frac{1}{4}\big[\alpha (k+6)+10\big]< \alpha (k+6)$, because $\alpha (k+6) - \frac{1}{4}\big[\alpha (k+6)+10\big] = \alpha (k+6)\times \frac{3}{4} -\frac{5}{2} \geq \frac{1}{2}\times(k+6)\times \frac{3}{4}-\frac{5}{2}=\frac{3k}{8} -\frac{1}{4}>0$ for $\alpha \in [\frac{1}{2}, 1)$ and $k\geq 1$. So
        \begin{align}\label{eq11}
            4\alpha + (1-\alpha)\frac{k+10}{4} < \alpha (k+6) + \frac{(1-\alpha)^2}{\alpha}, \qquad \alpha \in \Big[\frac{1}{2}, 1\Big).
        \end{align}
        And $\alpha + (1-\alpha)\frac{k+4}{1} < 2\alpha + (1-\alpha)(k+4)= \alpha (k+6) +(1-2\alpha)(k+4)\leq \alpha (k+6)$, when $\alpha \in [\frac{1}{2}, 1)$. Therefore
        \begin{align}\label{eq101}
            \alpha + (1-\alpha)\frac{k+4}{1} < \alpha (k+6) + \frac{(1-\alpha)^2}{\alpha}, \qquad \alpha \in \Big[\frac{1}{2}, 1\Big). 
        \end{align}
        Using \eqref{eq9}, \eqref{eq10}, \eqref{eq11} and \eqref{eq101} in \eqref{eq8}, we get
        \begin{align}\label{eq12}
            \rho_{\alpha}(\mathcal{T}_6) < \alpha (k+6) + \frac{(1-\alpha)^2}{\alpha}, \qquad \alpha \in \Big[\frac{1}{2}, 1\Big).
        \end{align}
        Claim $2$ now follows from \eqref{eq2} and \eqref{eq12}.
    \end{proof}
    \noindent\textbf{Claim 3:} For $\alpha \in [\frac{1}{2}, 1)$, $\rho_{\alpha}(\mathcal{T}_3) > \rho_{\alpha}(\mathcal{T}_4)$.
    \begin{proof}[\unskip\nopunct]\let\qed\relax
        \textbf{Proof of Claim 3:} \textit{Case 1: $\alpha =\frac{1}{2}$.}\par
        When $\alpha =\frac{1}{2}$, from Lemma \ref{lem2}, we have $\rho_{\frac{1}{2}}(\mathcal{T}_3) \geq \frac{k+6}{2} + \frac{1}{2}= \frac{k+7}{2}$. Using Proposition 12 from \cite{A_Q-merging_by_Nikiforov}, it can be seen that equality in the above relation is not possible, otherwise $\mathcal{T}_3$ would have been the graph $K_{1,k+6}$, which is a contradiction. Therefore $\rho_{\frac{1}{2}}(\mathcal{T}_3) > \frac{k+7}{2}$. If $\lambda(L_S(\mathcal{T}_4))$ is the spectral radius of $L_S(\mathcal{T}_4)$, as mentioned in the proof of Theorem $3.5$ in \cite{signless_lap_spec_rad_of_tricyclic_tree_pendant}, we have $\lambda (L_S(\mathcal{T}_4)) \leq k+7$. Therefore $\rho_{\frac{1}{2}}(\mathcal{T}_4) = \frac{1}{2}\lambda (L_S(\mathcal{T}_4)) \leq \frac{k+7}{2}$. Hence $\rho_{\alpha} (\mathcal{T}_3) > \rho_{\alpha}(\mathcal{T}_4)$, when $\alpha =\frac{1}{2}$.\par
        \noindent\textit{Case 2: $\alpha \in (\frac{1}{2}, 1)$.}\par
        When $\alpha \in (\frac{1}{2}, 1)$, $\rho_{\alpha}(\mathcal{T}_3) \geq \alpha (k+6) + \frac{(1-\alpha)^2}{\alpha}$ (from Lemma \ref{lem2}). Applying Lemma \ref{lem8} and using $\lambda (L_S(\mathcal{T}_4)) \leq k+7$, we get $\rho_{\alpha}(\mathcal{T}_4) \leq (1-\alpha)(k+7) +(2\alpha -1)(k+5) = \alpha (k+6) + (2-3\alpha) < \alpha (k+6) + \frac{(1-\alpha)^2}{\alpha}$, as $\frac{(1-\alpha)^2}{\alpha} - (2-3\alpha) =\frac{1-2\alpha +\alpha^2 -2\alpha +3\alpha^2}{\alpha} = \frac{(2\alpha -1)^2}{\alpha} >0$ when $\alpha \in (\frac{1}{2}, 1)$. Therefore $\rho_{\alpha} (\mathcal{T}_3) > \rho_{\alpha}(\mathcal{T}_4)$, when $\alpha \in (\frac{1}{2}, 1)$.\par
        \noindent Both the cases, together, make Claim $3$ true.
    \end{proof}
    Now the required result follows from Claims $1, 2$ and $3$.
\end{proof}

\begin{lemma}\label{r8}
    Let $\alpha \in [0,1)$. For each $i=3,4,6$ and $7$, let $G_i^*$ be the graph with largest $A_{\alpha}$-spectral radius in $\mathscr{T}_n^{k,i}$. Then the following statements hold true.
    \begin{enumerate}[label={\upshape (\alph*)}]
        \item There is exactly one vertex, say $v_i$, in $G_i^*$ which is attached to a pendant tree, say $T_i$.
        \item The tree $T_i$ (of \upshape(a)) attached to $v_i$ is consisting of $k$ pendant paths attached to $v_i$ in $G_i^*$. Moreover, these $k$ pendant paths have nearly equal lengths.  
    \end{enumerate}   
\end{lemma}

\begin{proof}
    We recall that $G_i^*$ is derived from some graph in Figure \ref{f1} by attaching trees to some of its vertices. Let $X$ be the $A_{\alpha}$-Perron vector of $G_i^*$.
    \begin{enumerate}[label={\upshape (\alph*)}]
        \item If possible, let there be two distinct vertices $v_a$ and $v_b$ in $G_i^*$, which are attached to pendant trees $T_a$ and $T_b$, respectively. Let $v_c$ and $v_d$ be vertices in $T_a$ and $T_b$, respectively, such that $d(v_a, v_c)= d(v_b, v_d) =1$. Without loss of generality, we assume $x_{v_a} \geq x_{v_b}$. Let $G_i^{**}= G_i^* -\{v_b v_d\} +\{ v_a v_d\}$. Then by Lemma \ref{lem3}, $\rho_{\alpha}(G_i^{**}) > \rho_{\alpha}(G_i^{*})$. This contradicts the maximality of $\rho_{\alpha}(G_i^*)$ as $G_i^{**} \in \mathscr{T}_n^{k,i}$. Hence, the result follows.
        \item If possible, let there be a vertex in $V(T_i) - \{v_i\}$ with degree at least three. Among all the vertices with degree at least three in $V(T_i) - \{v_i\}$, let $w$ be a vertex such that $d(v_i, w)$ is the minimum. Now two cases arise.
        \medskip\\ 
        \noindent\textit{Case 1: $d(v_i,w) =1$.} \par 
        First, let $x_{v_i} \geq x_w$. Since $d(w) \geq 3$, we find a vertex $u$ $(\neq v_i)$ adjacent to $w$ in $T_i$. We construct $G_i^{**}= G_i^* -\{ wu\} +\{ v_iu\}$. Then $G_i^{**} \in \mathscr{T}_n^{k,i}$, and by Lemma \ref{lem3}, $\rho_{\alpha}(G_i^{**}) > \rho_{\alpha}(G_i^{*})$, which is a contradiction. If $x_{v_i} < x_w$, then we take a vertex $y$ adjacent to $v_i$, which lies on a cycle containing $v_i$ in $G^*$. We form $G_i^{***} = G_i^* - \{ yv_i\} + \{ yw\}$. Then $G_i^{***} \in \mathscr{T}_n^{k,i}$. By Lemma \ref{lem3}, $\rho_{\alpha}(G_i^{***}) > \rho_{\alpha}(G_i^{*})$, and we get a contradiction again.
        \medskip\\         
        \noindent\textit{Case 2: $d(v_i,w) \geq 2$.}\par
        Let $P$ be the shortest $v_i - w$ path in $T_i$. Since, among all vertices in $V(T_i)-\{v_i\}$ with degree at least three, $w$ is nearest to $v_i$, and since $d(v_i,w) \geq 2$, $P$ is an internal path in $G_i^*$ of length at least two. Identifying any two adjacent vertices on $P$ and then subdividing a pendant edge of $T_i$ results in a new graph in $\mathscr{T}_n^{k,i}$ with strictly larger $A_{\alpha}$-spectral radius, by Lemma \ref{lem4} and Lemma \ref{lem6}. Therefore, we get a contradiction.\par
        Thus, the degree of every vertex in $V(T_i) - \{v_i\}$ is at most two. Since $v_i$ is the unique vertex in $G_i^*$, which is attached to a pendant tree, and the number of pendant vertices in $G_i^*$ is $k$, $T_i$ is consisting of $k$ pendant paths.\par
        Among the $k$ pendant paths attached
        to $v_i$, if possible, let there be two paths $P_j$ and $P_l$ with length of $P_j \geq$ length of $P_l$ + $2$. We construct a graph $G_i^{**}$ from $G_i^{*}$ by deleting the pendant vertex, say $z$, of $P_j$, and making $z$ adjacent to the pendant vertex of $P_l$. Then $G_i^{**} \in \mathscr{T}_n^{k,i}$. Now by Lemma \ref{lem5}, we have $\rho_{\alpha}(G_i^{**}) > \rho_{\alpha}(G_i^{*})$, a contradiction. Hence the $k$ pendant paths attached to $v_i$ have nearly equal lengths.
    \end{enumerate}
     This completes the proof of Lemma \ref{r8}.
\end{proof}

\begin{lemma}\label{r9}
    Let $\alpha \in [0,1)$. For each $i=3,4,6$ and $7$, let $G_i^*$ be the graph with largest $A_{\alpha}$-spectral radius in $\mathscr{T}_n^{k,i}$. Then for any internal path $P=v_0 v_1 \ldots v_k$ in $G_i^*$, we get
    \begin{enumerate}[label={\upshape (\alph*)}]
        \item If $v_0\neq v_k$, then the length of $P$ is at most $2$. Further, if the length of $P$ is exactly $2$, then $v_0$ and $v_k$ must be adjacent in $G_i^*$.
        \item If $v_0= v_k$, then the length of $P$ is exactly 3.
    \end{enumerate}
\end{lemma}

\begin{proof}
    \begin{enumerate}[label={\upshape (\alph*)}]
        \item If possible, let the length of $P$ be at least $3$. If we contract the edge $v_0 v_1$, and subdivide a pendant edge of $G_i^*$, the resultant graph will still be in $\mathscr{T}_{n}^{k,i}$, and by Lemma \ref{lem4} and Lemma \ref{lem6}, it will have strictly larger $A_{\alpha}$-spectral radius than that of $G_i^*$, a contradiction. Hence, the length $P$ is at most $2$.
        Next, let the length of $P$ be exactly $2$. If possible, let $v_0$ and $v_k$ be non adjacent. Then, by a similar argument as above, we obtain a contradiction, and hence the result follows.
        \item If possible, let the length of $P$ be greater than $3$. Then we contract the edge $v_0 v_1$, subdivide a pendant edge of $G_i^*$, proceed in the similar way as in the proof of part (a), and obtain a contradiction. \qedhere
    \end{enumerate}    
\end{proof}

\begin{lemma}\label{r2}
    Let $\alpha \in [0, 1)$ and $n,k$ be integers with $1\leq k \leq n-7$. Then among all the graphs in $\mathscr{T}_n^{k,7}$, $\mathcal{T}_7$ stands out as the sole graph to have the largest $A_{\alpha}$-spectral radius.
\end{lemma}

\begin{proof}
    The seven cycles within any graph of $\mathscr{T}_n^{k,7}$ can be configured in only one possible way as shown in Figure \ref{f3}. A graph in $\mathscr{T}_n^{k,7}$ is obtained by attaching trees to some vertices of the graph in the said figure. For convenience, $v_1, v_2, v_3, v_4$ are four vertices, and $P_{l_1 +1}, P_{l_2+1}, \ldots, P_{l_6+1}$ are six paths (of lengths $l_1, l_2, \ldots, l_6$, respectively) as shown in the figure. Let $G^* \in \mathscr{T}_n^{k,7}$ be the graph having the largest $A_{\alpha}$-spectral radius. We denote the $A_{\alpha}$-Perron vector of $G^*$ by $X$. From Lemma \ref{r8}, we know that $G^*$ is derived by attaching $k$ pendant paths of nearly equal lengths to one single vertex, say $v$, of the graph shown in the figure \ref{f3}. Therefore, to complete the proof of the lemma, it suffices to prove the following claim.\medskip\\
    \noindent\textbf{Claim $1$:} All the paths $P_{l_1 +1}, P_{l_2+1}, \ldots, P_{l_6+1}$ of $G^*$ given in Figure \ref{f3} are of length one.
    \begin{proof}[\unskip\nopunct]\let\qed\relax
        \textbf{Proof of Claim $1$:} As per \cite{spec_rad_of_tricyclic_by_geng_li}, each $l_i \geq 1$. We have to prove $l_1 =l_2=\ldots = l_6 =1$. If possible let $l_1 \geq 2$, i.e., $P_{l_1+1}$ is a path of length at least two connecting $v_1$ and $v_2$.
        \begin{figure}[t]
        \centering
        \includegraphics[scale=.7]{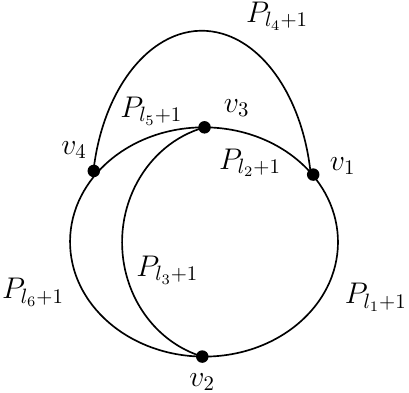}
        \caption{Only possible cycle configuration in graphs of $\mathscr{T}_n^{k,7}$.}
        \label{f3}
        \end{figure}
        Let $P_{l_1+1} = v_1 u_1 u_2 \ldots u_s$, where $u_s=v_2$ and $s \geq 2$. Note that $d(v_1) \geq 3$, $d(v_2) \geq 3$. $G^*$ has the vertex $v$ where the $k$ paths of nearly equal length are attached. Suppose $v \neq u_i$, $i = 1, 2, \ldots, s-1$. Then $P_{l_1+1}$ is an internal path. By Lemma \ref{r9}(a), length of $P_{l_1+1}$ i.e. $l_1$ can not be $2$ or more, a contradiction according to the initial assumption $l_1 \geq 2$. But if $v$ is some $u_t$, $1 \leq t \leq s-1$, then either at least one from the paths $v_1 u_1 \ldots u_t (=v)$ and $(v=)u_t u_{t+1} \ldots v_2$ will be internal path of length at least two in $G^*$ (in that scenario, using Lemma \ref{r9}(a), again we arrive at a contradiction), or $d(v_1, v) = d(v, v_2) = 1$. To discuss the later case, let $v_p$ ($\neq v$) be an adjacent vertex of $v_1$ in $G^*$, and $v_q$ be an adjacent vertex of $v$ on any of the pendant paths attached at $v$. We note that $v_p \neq v_2$, because $v_2$ lies on $P_{l_1+1}, P_{l_3+1}$ and $P_{l_6+1}$, whereas $v_p$ lies only on $P_{l_2+1}$ or $P_{l_4+1}$ (see Figure \ref{f3}). If $x_v \geq x_{v_1}$, we remove the edge $v_1 v_p$ and insert the edge $v v_p$ to $G^*$. The new graph will be in $\mathscr{T}_{n}^{k,7}$, and by Lemma \ref{lem3} this graph will have strictly larger $A_{\alpha}$-spectral radius than that of $G^*$, a contradiction. Otherwise if $x_v < x_{v_1}$, we remove $v v_q$ and insert $v_1 v_q$ to $G^*$. The new graph will be again in $\mathscr{T}_{n}^{k,7}$, and its $A_{\alpha}$-spectral radius is strictly greater than that of $G^*$ by Lemma \ref{lem3}, a contradiction again. Therefore $l_1 = 1$. Similarly we can show that $l_2 = l_3 =l_4 = l_5 = l_6 = 1$.
    \end{proof}
    This concludes that $G^* =\mathcal{T}_7$. 
\end{proof}

\begin{lemma}\label{r3}
    Let $\alpha \in [0, 1)$ and $n,k$ be integers with $1\leq k \leq n-7$. Then out of all graphs in $\mathscr{T}_n^{k,6}$, $\mathcal{T}_6$ stands out as the sole graph to have the largest $A_{\alpha}$-spectral radius.
\end{lemma}

\begin{proof}
    Since every $G \in \mathscr{T}_n^{k,6}$ has exactly $6$ cycles, the configuration of the cycles (excluding the possible trees attached) in $G$ will be one among $(a)$, $(b)$ or $(c)$ in Figure \ref{f4}. Let $G^* \in \mathscr{T}_n^{k,6}$ be having the largest $A_{\alpha}$-spectral radius among all graphs in $\mathscr{T}_n^{k,6}$. We denote the $A_{\alpha}$-Perron vector of $G^*$ by the column vector $X$. By Lemma \ref{r8}, it follows that $G^*$ is the graph obtained by attaching $k$ paths of nearly equal lengths to a unique vertex of the graph $(a)$, $(b)$ or $(c)$. Now we claim the following statement.
    \begin{figure}[t]
    \centering
    \includegraphics[scale=.7]{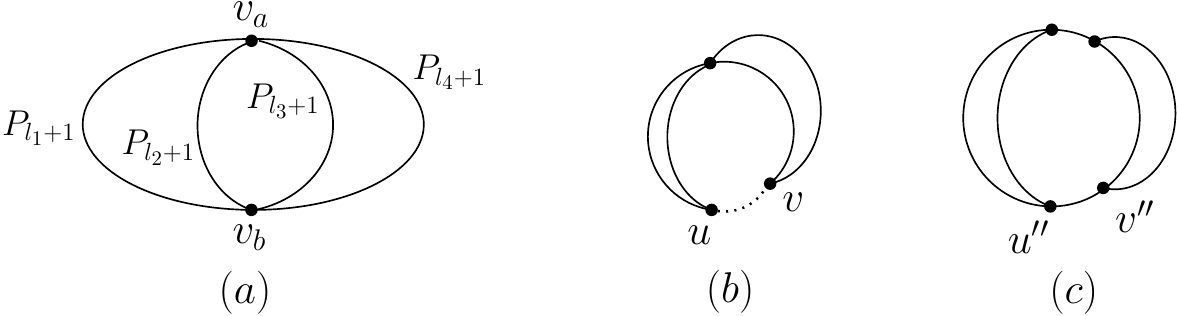}
    \caption{All possible configurations of six cycles in graphs of $\mathscr{T}_n^{k,6}$.}
    \label{f4}
    \end{figure}
    \medskip\\
    \textbf{Claim 1:} The cycle configuration in $G^*$ must be of type $(a)$ only.
    \begin{proof}[\unskip\nopunct]\let\qed\relax
        \textbf{Proof of Claim 1:} If possible let the configuration be type $(b)$. Let $u$ and $v$ be the vertices of $G^*$ as shown in Figure \ref{f4}$(b)$. We note that the dotted $u-v$ path in the said figure is the unique internal path from $u$ to $v$ in $(b)$. All other $u-v$ paths in $(b)$ are not internal paths, since each of them contains at least one internal vertex of degree greater than 2. Now three cases arise here.
        \medskip\\
        \textit{Case 1: Length of the internal path from $u$ to $v$ in $(b)$ is $1$.}\par
        Let $u^{\prime}$ ($\neq v$) and $v^{\prime}$ ($\neq u$) be the adjacent vertices of $u$ and $v$, respectively, in graph $(b)$. If $x_u \geq x_v$, then remove the edge $vv^{\prime}$ and insert the edge $uv^{\prime}$. The new graph will be still in $\mathscr{T}_n^{k,6}$, and by Lemma \ref{lem3}, its $A_{\alpha}$-spectral radius is strictly greater than that of $G^*$, which is a contradiction. If $x_u < x_v$, then removal of $uu^{\prime}$ and insertion of $vu^{\prime}$ will strictly increase the $A_{\alpha}$-spectral radius, which is a contradiction again.
        \medskip\\
        \textit{Case 2: Length of the internal path from $u$ to $v$ in $(b)$ (the path is not necessarily an internal path in $G^*$) is $2$.}\par
        We denote the internal path in $(b)$ by $P_3 = u u_1 v$. In this case, we might encounter two situations. Firstly, let $u_1$ be the vertex such that the $k$ pendant paths are attached at it. Let $u_2$ be a vertex adjacent to $u_1$ in some attached pendant path, and $v_1$ ($\neq u_1$) be a vertex adjacent to $v$ in $G^*$. If $x_{u_1} \geq x_v$, then after removing the edge $v v_1$, and adding the edge $u_1 v_1$ to $G^*$, the resultant graph will be still in $\mathscr{T}_n^{k,6}$. By Lemma \ref{lem3}, this new graph has strictly greater $A_{\alpha}$-spectral radius than that of $G^*$, a contradiction. If $x_{u_1} < x_v$, then removing $u_1 u_2$, and adding $v u_2$ to $G^*$ will increase the $A_{\alpha}$-spectral radius by Lemma \ref{lem3}, keeping the new graph still in $\mathscr{T}_n^{k,6}$, which is again a contradiction. Thus we are done with the first situation. On the other hand, if the pendant paths are not attached at $u_1$, then the internal path $P_3= u u_1 v$ in $(b)$ will be an internal path (of length $2$) in $G^*$ too, which leads to a contradiction according to Lemma \ref{r9}(a).
        \medskip\\
        \textit{Case 3: Length of the internal path from $u$ to $v$ in $(b)$ (not necessarily internal in $G^*$) is greater than $2$.}\par
        Whether $k$ pendant paths are attached to some vertex of the internal path in $(b)$, or not; we can always find a subpath of that path, which will be an internal path of length at least two in $G^*$. a contradiction according to Lemma \ref{r9}(a).\par
        Thus from the three cases discussed above, it is proved that configuration $(b)$ is not present in $G^*$. If possible let the configuration of the cycles be type $(c)$, and $u^{\prime \prime}$ and $v^{\prime \prime}$ be the vertices of $G^*$ in Figure \ref{f4}$(c)$. Then similarly as we did for $(b)$, we get a contradiction here also. This implies that $G^*$ can not have configuration $(c)$ too. Hence the only possible cycle configuration in $G^*$ is $(a)$. This concludes the proof of Claim 1.
    \end{proof}
     Now let $v_a$, $v_b$ be the two vertices, and $P_{l_1+1}, P_{l_2+1}, P_{l_3+1}, P_{l_4+1}$ be the four paths in $(a)$ as shown in Figure \ref{f4}. We claim the following statement.
    \medskip\\
    \textbf{Claim 2:} One of $l_1, l_2, l_3$ and $l_4$ is equal to $1$, and the other three are equal to $2$.
    \begin{proof}[\unskip\nopunct]\let\qed\relax
        \textbf{Proof of Claim 2:} If possible let $l_1 \geq 3$. Then irrespective of the location of the vertex, at which the $k$ paths of nearly equal lengths are attached in $(a)$ to make it $G^*$, we can always find an internal path (which is actually a subpath of $P_{l_1+1}$) in $G^*$ of length at least two, contradicting Lemma \ref{r9}(a). Hence $l_1 \leq 2$. $P_{l_i+1}$'s can not be paths of length zero as per \cite{spec_rad_of_tricyclic_by_geng_li}. Therefore $l_1 =1$ or $2$. In the similar way, we verify $l_i =1$ or $2$, $i=2, 3, 4$. Now if $l_1 =l_2 =l_3 =l_4 =2$, then we can always choose a path among those four paths, which is an internal path of length two in $G^*$, leading to a contradiction, as per Lemma \ref{r9}(a). Therefore at least one of $l_1, l_2, l_3$ and $l_4$ must be equal to $1$. Now any two of the six cycles in $(a)$ have two vertices $v_a$ and $v_b$ in common, so the number of paths of length one among $P_{l_1+1}, P_{l_2+1}, P_{l_3+1}, P_{l_4+1}$ is at most one, otherwise the graph would have had multiple edges. Hence exactly one of the four paths is of length one, and the other three are of length two.
    \end{proof}
    Assessing all the information gathered, we see that we are now left with only two possible candidates for $G^*$; one is obtained by attaching $k$ paths of nearly equal lengths to vertex $u_a$, and the other by attaching $k$ paths of nearly equal lengths to vertex $u_b$ of the graph shown in Figure \ref{f5} (the later graph is $\mathcal{T}_6$, as mentioned earlier). If possible let $G^*$ be the graph obtained by attaching the paths to $u_a$. Suppose $u_c$ is an adjacent vertex of $u_a$ on one of the pendant paths attached at $u_a$, and $u_d$ ($\neq u_a$, also not adjacent to $u_a$) is an adjacent vertex of $u_b$ in $G^*$. If $x_{u_a} \geq x_{u_b}$, we remove the edge $u_b u_d$ and insert the edge $u_a u_d$ to $G^*$. The new graph will be in $\mathscr{T}_{n}^{k,6}$, and by Lemma \ref{lem3} this graph has strictly larger $A_{\alpha}$-spectral radius than that of $G^*$, a contradiction. On the other hand, if $x_{u_a} < x_{u_b}$, we remove $u_a u_c$ and insert $u_b u_c$ to $G^*$. The new graph will be again in $\mathscr{T}_{n}^{k,6}$, and its $A_{\alpha}$-spectral radius is strictly greater than that of $G^*$ by Lemma \ref{lem3}, a contradiction again. Hence $G^*$ is the graph obtained by attaching $k$ paths of nearly equal lengths to vertex $u_b$ of the graph in Figure \ref{f5}, i.e., $G^* = \mathcal{T}_6$. This brings the proof to its conclusion.\begin{figure}[t]
    \centering
    \includegraphics[scale=.7]{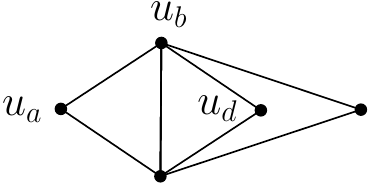}
    \caption{Final configuration of six cycles in $G^*$.}
    \label{f5}
    \end{figure}
\end{proof}

\begin{lemma}\label{r4}
    Let $\alpha \in [0, 1)$ and $n,k$ be integers with $1\leq k \leq n-7$. Then out of all graphs in $\mathscr{T}_n^{k,4}$, $\mathcal{T}_4$ stands out as the sole graph to have the largest $A_{\alpha}$-spectral radius.
\end{lemma}

\begin{proof}
    Consider the graph in Figure \ref{f6}$(a)$ which consists of three internally disjoint paths $P_{k_1+1}, P_{k_2+1}$ and $P_{k_3+1}$ with $k_1,k_2,k_3 \geq 1$ and at most one of $k_1,k_2$ and $k_3$ is exactly $1$. Then connecting this graph with a cycle $C_p$ of length $p$ by a path $P_{q+1}$ ($q\geq 0$), we obtain four possible graphs as given in Figure \ref{f6}$(b)-(e)$. A graph in $\mathscr{T}_n^{k,4}$ can be obtained by attaching trees to vertices of one of these four types of graphs. Basically \ref{f6}$(b)-(e)$ are the types of configuration of the four cycles permitted to be in any graph of $\mathscr{T}_n^{k,4}$. Let $G^* \in \mathscr{T}_n^{k,4}$ be a graph with largest possible $A_{\alpha}$-spectral radius. Like before, we denote the $A_{\alpha}$-Perron vector of $G^*$ by $X$. In the rest of the proof, by $(b)$, $(c)$, $(d)$ and $(e)$ we mean the graphs given respectively in \ref{f6}$(b)-(e)$.\par
    \begin{figure}[t]
    \centering
    \includegraphics[scale=.7]{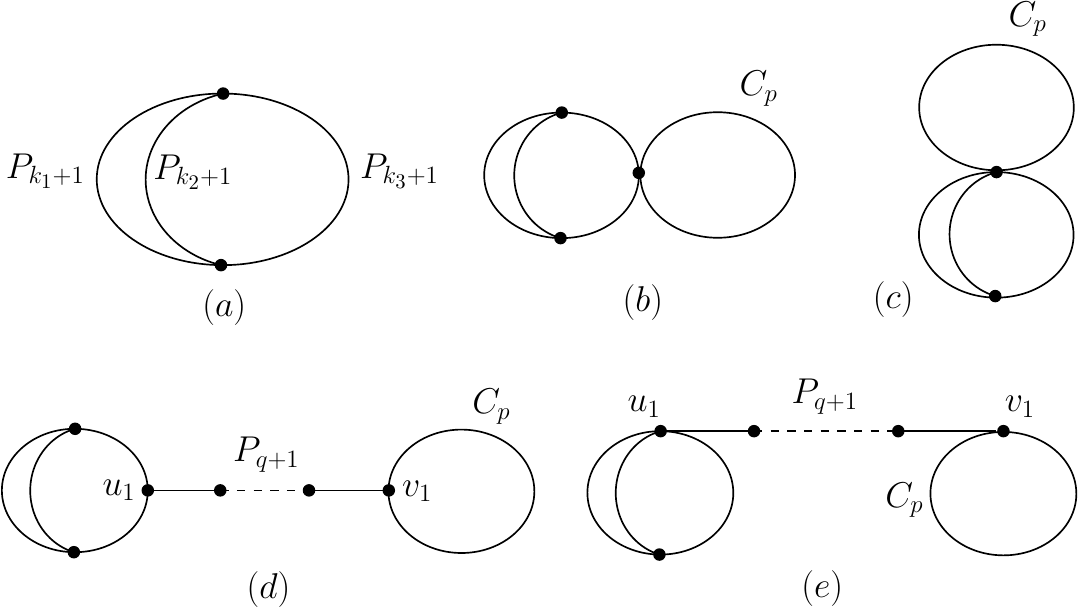}
    \caption{$(a)$ is graph derived from identifying respective initial and terminal vertices of $P_{k_1+1}, P_{k_2+1}$ and $P_{k_3+1}$, and $(b), (c), (d)$ and $(e)$ are the four possible cycle configurations in graphs of $\mathscr{T}_n^{k,4}$.}
    \label{f6}
    \end{figure}
    Applying Lemma \ref{r8}, it follows that $G^*$ is obtained by attaching $k$ pendant paths of nearly equal lengths to a unique vertex of either $(b)$, $(c)$, $(d)$ or $(e)$. Now if possible let the configuration of cycles in $G^*$ is of type $(d)$. For convenience, let $u_1$ and $v_1$ are the end vertices of $P_{q+1}$ in $(d)$ as shown in Figure \ref{f6}. Then three cases might occur: $(i)$ $d(u_1,v_1) =1$, $(ii)$ $d(u_1,v_1) =2$ and $(iii)$ $d(u_1,v_1) \geq 3$. Using the technique we used in the proof of Claim $1$ in Lemma \ref{r3}, we encounter contradiction in each of the three cases. Therefore it is not possible to have $(d)$ type of configuration of cycles in $G^*$. Similarly, we can prove that configuration $(e)$ is also not permitted in $G^*$. Hence cycles in $G^*$ are of type $(b)$ or $(c)$.\par
    Now we will prove that $C_p$ of $(b)$ (respectively, $(c)$) in $G^*$ is of length exactly three. Since $G^*$ is a simple graph, $C_p$ has length at least three. If possible let it has length four or more. Then irrespective of the position of the attached pendant paths in $G^*$, we can always find a suitable internal path of length at least $2$ on $C_p$, which is a contradiction to Lemma \ref{r9}(a). This implies that $C_p$ has length exactly three.\par
    We denote the vertex to which $C_p$ is attached with the rest of graph $(b)$ (respectively, with the rest of graph $(c)$) in $G^*$ as $v_r$. We denote the other two vertices of $C_p$ as $v_s$ and $v_t$. Let the pendant paths be attached to the vertex $v$ in $(b)$ (respectively, in $(c)$). We will prove that $v=v_r$. If possible let $v \neq v_r$. Then $v$ is either one of $v_s$ and $v_t$, or it is any vertex of $(b)$, except $v_r, v_s$ or $v_t$. First, we consider the case when $v$ is $v_s$ (respectively, $v_t$). If $x_{v_r} \geq x_v$, then we shift the pendant paths attached at vertex $v$ to the vertex $v_r$. The new graph will be still in $\mathscr{T}_n^{k,4}$, and by Lemma \ref{lem3}, its $A_{\alpha}$-spectral radius is strictly greater than that of $G^*$, a contradiction. If $x_{v_r} < x_v$, then we remove the edges between $v_r$ and its all neighbors (except $v_s$ and $v_t$), and then add edges between $v$ and all those neighbors (except $v_s$ and $v_t$) of $v_r$ in $G^*$. The resultant graph will be in $\mathscr{T}_n^{k,4}$, and by Lemma \ref{lem3}, its $A_{\alpha}$-spectral radius is strictly greater than that of $G^*$, again a contradiction. The second case is when $v$ is any vertex of $(b)$ other than $v_r, v_s$ or $v_t$. If $x_{v_r} \geq x_v$, then we shift the pendant paths attached at vertex $v$ to the vertex $v_r$. The new graph will be still in $\mathscr{T}_n^{k,4}$, and by Lemma \ref{lem3}, its $A_{\alpha}$-spectral radius is strictly greater than that of $G^*$, a contradiction. On the other hand, if $x_{v_r} < x_v$, we construct $G^{**} = G^* -\{ v_r v_s, v_r v_t\}+\{v v_s, v v_t\}$. $G^{**} \in \mathscr{T}_n^{k,4}$, and by Lemma \ref{lem3}, $\rho_{\alpha}(G^{**}) > \rho_{\alpha}(G^{*})$, a contradiction. Therefore it is proved that $v = v_r$.\par
    Now we prove that the cycles in $G^*$ must be of type $(c)$. If possible let us assume they are of type $(b)$. Let in $G^*$, $P_{k_1+1} = v_a \ldots v_r \ldots v_b$ ($k_1 \geq 2$) be a path of length $k_1$ in $(b)$, see Figure \ref{f7}. We note that $v_r$ ($\neq v_a, v_b$) is the vertex to which $C_p$ and $k$ pendant paths are attached. We choose the path of length at least two from the paths $P_{k_2+1}$ and $P_{k_3+1}$, without loss of generality, say it is $P_{k_2+1}$. Let $v_c$ ($\neq v_b$) be adjacent to $v_a$ on $P_{k_2+1}$. If $x_{v_a} \geq x_{v_r}$, then we shift the pendant paths from vertex $v_r$ to vertex $v_a$. The resultant graph will be in $\mathscr{T}_n^{k,4}$, and by Lemma \ref{lem3}, its $A_{\alpha}$-spectral radius is strictly greater than that of $G^*$, a contradiction. If $x_{v_a} < x_{v_r}$, we construct $G^{***} = G^* -\{v_a v_c\} + \{v_r v_c\}$. Then $G^{***} \in \mathscr{T}_n^{k,4}$, and by Lemma \ref{lem3}, $\rho_{\alpha}(G^{***}) > \rho_{\alpha}(G^{*})$, a contradiction. Hence the configuration of cycles in $G^*$ can be of type $(c)$ only.\par
    Finally, using the similar argument we used to prove Claim 2 in Lemma \ref{r3}, here, we can verify that for the paths $P_{k_1+1}, P_{k_2+1}$ and $P_{k_3+1}$, exactly one from $k_1, k_2$ and $k_3$ is equal to $1$, and the other two are equal to $2$.\par
    \begin{figure}[t]
    \centering
    \includegraphics[scale=.7]{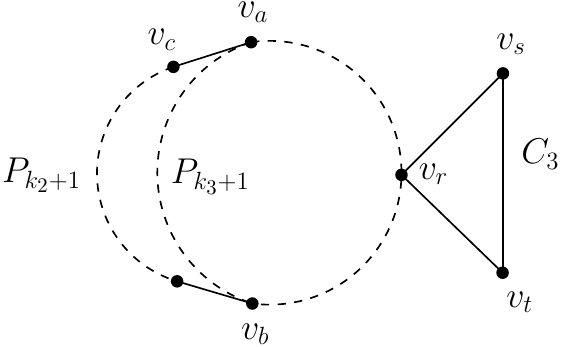}
    \caption{Type $(b)$ configuration of four cycles in $G^*$.}
    \label{f7}
    \end{figure}
    Combining all the above facts regarding the structure of $G^*$, we conclude that $G^* = \mathcal{T}_4$. This wraps up the proof.
\end{proof}

\begin{lemma}\label{r5}
    Let $\alpha \in [\frac{1}{2}, 1)$ and $n,k$ be integers with $1\leq k \leq n-7$. Then out of all graphs in $\mathscr{T}_n^{k,3}$, $\mathcal{T}_3$ stands out as the sole graph to have the largest $A_{\alpha}$-spectral radius.
\end{lemma}

\begin{proof}
    The configuration of three cycles, say $C_r, C_s$ and $C_t$ ($r,s,t \geq 3$), in any graph of $\mathscr{T}_n^{k,3}$, has seven possible cases as given in Figure \ref{f8}. A graph in $\mathscr{T}_n^{k,3}$ is obtained by attaching trees to some vertices of one of these seven graphs. Among all graphs in $\mathscr{T}_n^{k,3}$, let $G^*$ be having the largest $A_{\alpha}$-spectral radius. We denote the $A_{\alpha}$-Perron vector of $G^*$ by $X$.\par
    \begin{figure}[t]
    \centering
    \includegraphics[scale=.7]{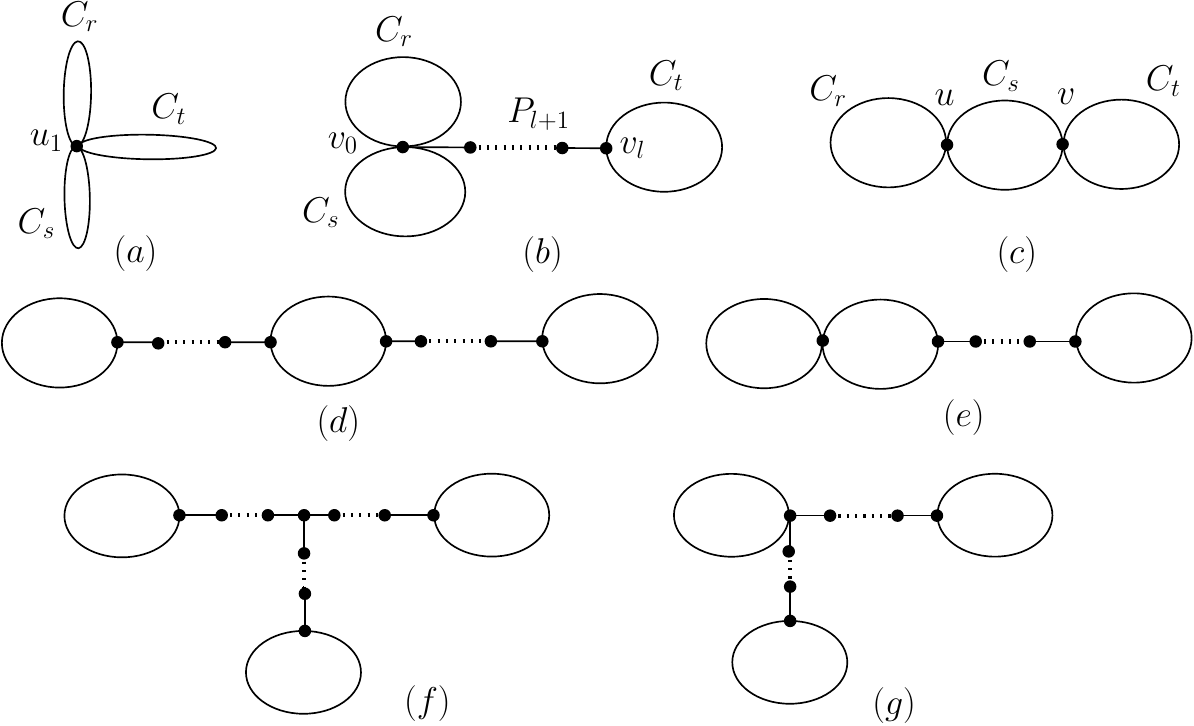}
    \caption{Seven possible configurations of three cycles in graphs of $\mathscr{T}_n^{k,3}$.}
    \label{f8}
    \end{figure}
    We first prove that the configuration of the cycles $C_r, C_s$ and $C_t$ in $G^*$ must be of type $(a)$ in Figure \ref{f8}. If possible let the configuration be of type $(b)$. In $(b)$, let $P_{l+1}=v_0 v_1 \ldots v_l$ ($l \geq 1$) be the path which connects two cycles, say $C_r$ and $C_t$, where $v_0$ is on $C_r$, and $v_l$ is on $C_t$. Also let $v_0^{\prime}$ be a neighbor of $v_0$ on $C_r$, and $v_l^{\prime}$ be a neighbor of $v_l$ on $C_t$. If $x_{v_0} \geq x_{v_l}$, we construct $G_1^* = G^* -\{ v_l v_l^{\prime}\} +\{v_0 v_l^{\prime} \}$. Then $G_1^* \in \mathscr{T}_n^{k,3}$, and by Lemma \ref{lem3}, $\rho_{\alpha} (G_1^*) > \rho_{\alpha}(G^*)$, a contradiction. On the other hand, if $x_{v_0} < x_{v_l}$, we construct $G_2^* = G^* -\{ v_0 v_0^{\prime}\} +\{v_l v_0^{\prime} \}$. Then $G_2^*$ is also in $\mathscr{T}_n^{k,3}$. By Lemma \ref{lem3}, $\rho_{\alpha} (G_2^*) > \rho_{\alpha}(G^*)$, a contradiction. Hence cycles in $G^*$ can not be arranged like $(b)$. Similarly we can show that the cycle configuration can not be of type $(d)$, $(e)$, $(f)$ or $(g)$. Now, if possible let the configuration be of type $(c)$. For convenience, let $u$ and $v$ respectively be the common vertices of $C_r$ and $C_s$, and $C_s$ and $C_t$, as shown in Figure \ref{f8}$(c)$. Let $u^{\prime}$ be a neighbor of $u$ in $C_r$, and $v^{\prime}$ be a neighbor of $v$ in $C_t$. Let $x_u \geq x_v$. We construct $G_3^*= G^* -\{v v^{\prime}\} +\{u v^{\prime}\}$. Then by Lemma \ref{lem3}, $\rho_{\alpha} (G_3^*) > \rho_{\alpha}(G^*)$, but interestingly $G_3^* \in \mathscr{T}_n^{k,4}$. Applying Lemma \ref{r4} here, we get $\rho_{\alpha} (\mathcal{T}_4) \geq \rho_{\alpha}(G_3^*)$. So $\rho_{\alpha} (\mathcal{T}_4) > \rho_{\alpha}(G^*)$. From Lemma \ref{r1}, we already know that $\rho_{\alpha} (\mathcal{T}_3) > \rho_{\alpha}(\mathcal{T}_4)$, which finally implies $\rho_{\alpha} (\mathcal{T}_3) > \rho_{\alpha}(G^*)$. Since $\mathcal{T}_3$ is in $\mathscr{T}_n^{k,3}$, we arrive at a contradiction. If $x_u < x_v$, then in the similar way as above, we reach to a contradiction. Therefore it is not possible to have configuration $(c)$ in $G^*$. Hence the configuration of the cycles in $G^*$ must be of type $(a)$.\par
    Let $u_1$ be the vertex common to all three cycles in $G^*$ as shown in Figure \ref{f8}$(a)$. From Lemma \ref{r8}, it follows that $G^*$ is obtained by attaching $k$ pendant paths of nearly equal lengths to a unique vertex of $(a)$. This vertex is none other than $u_1$, which we will justify now. If possible let the pendant paths be attached to a vertex $u_2$ ($\neq u_1$) on a cycle, say $C_r$. On the attached paths, we take a vertex $u_3$, which is adjacent to $u_2$. On a cycle other than $C_r$, say $C_s$, we take a vertex $u_4$ which is adjacent to $u_1$. Now if $x_{u_1} \geq x_{u_2}$, then the process of removal of the edge $u_2 u_3$ and then addition of the edge $u_1 u_3$ to $G^*$ increase the $A_{\alpha}$-spectral radius by Lemma \ref{lem3}, keeping the resultant graph still in $\mathscr{T}_n^{k,3}$, a contradiction. If $x_{u_1} < x_{u_2}$, then removing the edge $u_1 u_4$ and adding the edge $u_2 u_4$ will increase the $A_{\alpha}$-spectral radius by Lemma \ref{lem3}, although the resultant graph will be still in $\mathscr{T}_n^{k,3}$, leading to a contradiction again. Hence $G^*$ is obtained by attaching $k$ pendant paths of nearly equal lengths to the vertex common to all three cycles in $(a)$. Further, using Lemma \ref{r9}(b), we can confirm that the three cycles $C_r, C_s$ and $C_t$ in $G^*$ are of length exactly three each.\par
    Thus we can conclude that $G^* = \mathcal{T}_3$, and hence the result follows. 
\end{proof}

\begin{proof}[Proof of Theorem \ref{r6}]
    Combining Lemma $\ref{r1}$ and Lemmas $\ref{r2}-\ref{r5}$, the final result follows directly.
\end{proof}


\section{Concluding remarks}
We proved Lemmas $\ref{r2}-\ref{r4}$ for all values of $\alpha \in [0,1)$. However, we could not do so for Lemma \ref{r1}; we have proved the result therein for $\alpha \in [\frac{1}{2},1)$. Proof of Lemma \ref{r5} uses Lemma \ref{r1}, thus Lemma \ref{r5} has been proved for $\alpha \in [\frac{1}{2},1)$. So if one succeeds to prove that the statement of Lemma \ref{r1} is true for $\alpha \in [0,\frac{1}{2})$ also, then the general version of Theorem \ref{r6} can be obtained for $\alpha \in [0,1)$. Moreover, It can be noted that the problem of characterizing the graphs with largest $A_{\alpha}$-spectral radius ($\alpha \in [0,1)$) in the whole class of tricyclic graphs is still open.


\bigskip \noindent
\textbf{\Large Acknowledgment:} The authors are thankful to the anonymous referee for the valuable comments and suggestions which improved the presentation of the paper.

\section*{Statements and Declarations} 

\textbf{Competing interests:} The authors disclose that they have no conflicting interest.\par\noindent
\textbf{Data availability:} The research described in the article did not utilize any data.\par\noindent
\textbf{Funding information:} There is no funding source.


\bibliographystyle{plain}
\bibliography{bibliotri}
\end{document}